\documentclass[12pt, reqno]{amsart}
\usepackage{}
\usepackage{amsmath, amsthm, amscd, amsfonts, amssymb, graphicx, color,cite}
\usepackage[bookmarksnumbered, colorlinks, plainpages]{hyperref}

\newtheorem{thm}{Theorem}[section]
\newtheorem{cor}[thm]{Corollary}
\newtheorem{lem}[thm]{Lemma}
\newtheorem{prop}[thm]{Proposition}
\theoremstyle{definition}
\newtheorem{defn}{Definition}[section]
\theoremstyle{remark}

\numberwithin{equation}{section}

\begin{document}

\title[$L_p$ John ellipsoids for log concave functions]{$L_p$ John ellipsoids for log-concave functions}

\author[ F. Chen,  J. Fang, M. Luo, C. Yang]{Fangwei Chen$^1$, Jianbo Fang$^1$, Miao Luo$^2$, Congli Yang$^2$}

\address{1. School of Mathematics and Statistics, Guizhou University of Finance and Economics,
Guiyang, Guizhou 550025, People's Republic of China.}
\email{cfw-yy@126.com}
\email{16995239@qq.com}

\address{2. School of Mathematical Sciences, Guizhou Normal
University, Guiyang, Guizhou 550025, People's Republic of China}
\email{lm975318@126.com}
\email{yangcongli@gznu.edu.cn}

\thanks{The  work is supported in part by CNSF (Grant No. 11561012, 11861004, 11861024), Guizhou  Foundation for Science and Technology (Grant
No. [2019] 1055, [2019]1228), Science and technology top talent support program of Guizhou Eduction Department (Grant No. [2017]069).}




\dedicatory{}



\begin{abstract}
The aim of this paper is to develop the $L_p$ John ellipsoid for the geometry of log-concave functions. Using the results of the $L_p$ Minkowski theory for log-concave function established in \cite{fan-xin-ye-geo2020}, we characterize the $L_p$ John ellipsoid for log-concave function, and establish some inequalities of the $L_p$ John  ellipsoid for log-concave function. Finally, the analog of Ball's volume ratio inequality for the $L_p$ John ellipsoid of log-concave function is established.
\end{abstract} \maketitle

\section{introduction}

Let $K$ be a convex body in $\mathbb R^n$, among all ellipsoids contained in $K$, there exists a unique ellipsoid $JK$
 with the maximum volume, this ellipsoid is called the John's ellipsoid of $K$. It plays an important role in convex geometry and Banach space geometry (see, e.g., \cite{bal-vol1991,bal-ell1992,gru-joh2011,kla-on2004,bal-vol1989,gia-mil-ext2000,lew-ell1979,pis-the1989}). One of the most important results concerning the John ellipsoid is the Ball's volume-ratio inequality, which states that: if $K$ is an origin symmetric convex body in $\mathbb R^n$, then
 \begin{align*}
  \frac{V(K)}{V(JK)}\leq \frac{2^n}{\omega_n},
\end{align*}
with equality if and only $K$ is a parallelotope. Here $V(K)$ denotes the $n$-dimensional volume and $\omega_n=\frac{\pi^\frac{n}{2}}{\Gamma(1+\frac{n}{2})}$ denotes the volume of a unit ball in $\mathbb R^n$.

In l990's, the $L_p$ Brunn-Minkowski theory was firstly initiated by Lutwak (see \cite{lut-the1993,lut-the1996}), during the last two decades, it has achieved great development and expanded rapidly (see, e.g., \cite{cho-wan-the2006,fle-gue-pao-a2007,hab-lut-yan-zha-the2010,hab-sch-gen2009,
lud-gen2010,lut-yan-zha-lp2000,wer-ren2012,wer-ye-new2008,sch-wer-sur2004,sta-the2002,sta-cen2012,pao-wer-rel2012}).
The $L_p$ extension of the John ellipsoid is given by Lutwak, Yang  and Zhang \cite{lut-yan-zha-lp2005}.

Given a smooth convex body $K\in \mathbb R^n$ that contains the origin in its interior. Let $f_p(K,\cdot)$  be the $L_p$ curvature function of $K$, $p>0$, find
\begin{align}
  \min_{\phi\in SL(n)}\int_{S^{n-1}}f_p(\phi K,\cdot)dS(u),
\end{align}
where $S^{n-1}$ denotes the unit sphere in $\mathbb R^n$. The minimum is actually attained at some $\phi_p\in SL(n)$, and defines an ellipsoid $E_pK$, which $\phi_p$ maps it into the unit ball $B$, that is, $\phi_p E_pK=B$. The ellipsoid is unique and is called the volume-normalized $L_p$ John ellipsoid of $K$. The equivalent ways to state the above problem is given by the following two Optimization Problems \cite{lut-yan-zha-lp2005}: Given a convex body $K$ in $\mathbb R^n$ that contains the origin in it interior, find an ellipsoid $E$, amongst all origin-centered ellipsoids, which solves the following constrained maximization problem:
\begin{align}
  \max\Big(\frac{V(E)}{\omega_n}\Big)^{\frac{1}{n}} \,\,\,\,\,\mbox{ subject to}\,\,\,\, \overline V_p(K,E)\leq 1.
\end{align}
A maximal ellipsoid will be called an $S_p$ solution for $K$. The dual problem is to find $E$ such that
\begin{align}
  \min\overline V_p(K,E) \,\,\,\,\,\mbox{ subject to}\,\,\,\, \Big(\frac{V(E)}{\omega_n}\Big)^{\frac{1}{n}}\geq 1.
\end{align}
A minimal ellipsoid will be called an $\overline S_p$ solution for $K$.
Where  $$\overline V_p(K,E)=\Big[\int_{S^{n-1}}\Big(\frac{h_E}{h_K}\Big)^pd\overline V_K\Big]^{\frac{1}{p}},\,\,\,\,\,\,\,\,\,p>0,$$
is the normalized $L_p$ mixed volume of $K$ and $E$. More details about the solution of the two problems $S_p$, $\overline S_p$ and related inequalities see \cite{lut-yan-zha-lp2005}. The Orlicz extension of the John ellipsoid is done by Zou and Xiong \cite{zou-xio-orl2014}.
Recently, the study of the geometry of log-concave functions in the field of convex geometry has emerged, with a quite natural idea is to replace the volume of a convex body by the integral of a log-concave function. To establish functional version of the problems from the convex geometric analysis of convex body has attracted a lot of authors interest (see, e.g., \cite{cag-fra-gue-leh-sch-wer-fun2016,cag-wer-div2014,cag-wer-mix2015,
cag-ye-aff2016,col-bru2005,col-fra-the2013,gut-a2019,art-flo-seg-fun2020,
alo-mer-jim-vil-joh2018,rot-sup2013,lin-aff2017,cal-ye-aff2016,alo-mer-jim-vil-rog2016,alo-ber-mer-an2020,
alo-a2019}).   Also an extension of the John ellipsoids to the case of log-concave functions has attracted a lot of authors interest, for example, in \cite{gut-mer-jim-vil-joh2018}, the authors extend the notion of John's ellipsoid to the setting of integrable log-concave functions and obtain integral ratio of a log-concave function and establish the reverse functional affine isoperimetric inequality. The extension of the LYZ ellipsoid to the log-concave functions is done by Fang and Zhou \cite{fan-zho-lyz2018}. The L\"{o}wner ellipsoid function for log-concave function is invested by Li, Sch\"utt and Werner \cite{li-sch-wer-the2019}. Extensive research has been devoted to extend the concepts and inequalities from convex bodies to the setting of log-concave functions (see, e.g., \cite{era-kla-mom2015,fra-mey-som2007}). In fact, it was observed that the Pr\'ekopa-Leindler inequality is the functional analog of the Brunn-Minkowski inequality (see e.g., \cite{gar-the2002,bra-lie-on1976,pre-new1975}) for convex bodies. Much progress has been made see \cite{avi-kla-mil-the2004,avi-kla-sch-wer-fun2012,avi-mil-the2009,
col-fun2006}.

Let $f$ be a log-concave functions of $\mathbb R^n$ such that
\begin{align*}
  f:\mathbb R^n\rightarrow \mathbb R, \,\,\,\,f=e^{-u},
\end{align*}
where $u: \mathbb R^n\rightarrow \mathbb R \cup\{+\infty\}$ is a convex function. We always consider in this paper that a log-concave function $f$ is integrable and such that $f$ is nondegenerate, i.e., the interior of the support of $f$ is non-empty, $int(suppf)\neq \emptyset$. This implies that $0< \int_{\mathbb R^n}fdx< \infty$.

Let $f=e^{-u}$, $g=e^{-v}$ be log-concave functions, for any real $\alpha,\,\,\beta>0$, the Asplund sum and scalar multiplication of two log-concave functions are defined as,
\begin{align}
  \alpha \cdot f \oplus \beta\cdot g:=e^{-w}, \,\,\,\,\,\,\,\,\mbox{ where}\,\,\, w^*=\alpha u^*+\beta v^*.
\end{align}
Here $w^*$ denotes as usual the Fenchel conjugate of the convex function $\omega$. Correspond to the volume $V(K)$ of a convex body $K$ in $\mathbb R^n$, the total mass $J(f)$ of a log-concave function $f$  in $\mathbb R^n$ is firstly considered in \cite{col-fra-the2013}. The functional counterpart of Minkowski's first inequality and related isoperimetric inequalities are established. The $L_p$ extension of the  Asplund sum and scalar multiplication of two log-concave functions are discussed in \cite{fan-xin-ye-geo2020}, the functional $L_p$ Minkowski's first inequality and the functional $L_p$ Minkowski problem also been discussed.

Our main goals in this paper are to discuss the functional $L_p$ John ellipsoid, based on the $L_p$ Asplund sum and $L_p$ scalar mutiplication of two log-concave functions. Owing to the functional $L_p$ Minkowski's first variation of $f$ and $g$, we focus on the following:

\textbf{Problem $S_p$.} Given a log-concave function $f\in\mathcal A_0,$ find a Gaussian function $\gamma_\phi $ which solves the following constrained maximization problem:
\begin{align}
  \max\Big(\frac{J(\gamma_\phi)}{c_n}\Big)\,\,\,\,\mbox{ subject to}\,\,\,\,\,\,   \overline\delta J_p\big(f,\gamma_\phi\big)\leq 1.
\end{align}
Where $c_n=(2\pi)^{\frac{n}{2}}$ and $\phi\in GL(n)$, $\gamma_\phi=e^{-\frac{\|\phi x\|^2}{2}}$ is the Gaussian function. $\overline\delta J_p\big(f,\gamma_\phi\big)$ is the normalized first variation of the total mass $J(f)$ with respect to the $L_p$ Asplund sum.

In section 3, we prove that there exists a unique Gaussian function which solves the Problem $S_p$. The unique Gaussian function which solves the problem $S_p$ is called the $L_p$ John ellipsoid for the log-concave function $f$, and denoted by $E_pf$. Moreover, we characterize a Gaussian function which is the solution of the problem $S_p$.

In section 4, we focus on the continuity of the $L_p$ John ellipsoid, we prove that the $L_p$ John ellipsoid $E_pf$ is continuous with respect to $f$ and $p$. By the $L_p$ Minkowski's first inequality for log-concave function, we prove that the total mass of the $L_p$ John ellipsoid $E_pf$ is no more than the total mass of $f$. In the end of this section,  we show the similar Ball's volume ration inequality is also holds for log-concave function.

\section{Preliminaries}

\subsection{Convex bodies}

In this paper, we work in $n$-dimensional Euclidean space, $\mathbb R^n$, endowed with the usual scalar product $\langle x, y\rangle$ and norm $\|x\|$. Let $B^n=\{x\in \mathbb R^n:\|x\|\leq 1\}$ denote the standard unit ball and $S^{n-1}=\{x\in \mathbb R^n:\|x\|= 1\}$ denote the unit sphere in $\mathbb R^n$.  Let $\mathcal K^n$ denote the class of convex bodies in $\mathbb R^n$, and $\mathcal K^n_o$  be the subclass of convex bodies $K$ whose relative interior $int(K)$ is nonempty.  For $i\leq n$, let $\mathcal H^i$ be the $i$-dimensional Hausdorff measure, we indicate by $V(K)=\mathcal H^n(K)$ the $n$-dimensional volume.

Let $h_K(\cdot):\mathbb R^{n}\rightarrow \mathbb R$ be the support function of $K$; i.e., for $x\in \mathbb R^{n}$,
$$h_K(x)=\max\big\{\langle x, y\rangle :y\in K\big\}.$$
 Let $n_K(x)$ be the unit outer normal at $x\in \partial K$, then $h_K(n_K(x))=\langle n_K(x), x\rangle.$
It is shown that the sublinear support function characterizes a convex body and, conversely, every sublinear function on $\mathbb R^n$ is the support function of a nonempty compact convex set. By the definition of the support function, if $\phi\in GL(n)$, then the support function of the image $\phi K:=\{\phi y:y\in K\}$ is given by
\begin{align*}
h_{\phi K}(x)=h_K(\phi^tx),
\end{align*}
where $\phi^t$ denotes the transpose of $\phi$. Let $K\in \mathcal K^n_o$ be a convex body that contains the origin in its interior, the polar body $K^\circ$ is defined by
\begin{align*}
  K^\circ=\big\{y\in \mathbb R^n:\langle y, x\rangle\leq 1, \,\,{\rm for \,\,\,all}\,\, x\in K\big\}.
\end{align*}
Obviously, for $\phi\in GL(n)$, then  $(\phi K)^\circ=\phi^{-t}K^\circ$.
The gauge function $\|\cdot\|_K$ is defined by
\begin{align*}
  \|x\|_K=\min\big\{a\geq 0: x\in \alpha K\big\}=\max_{y\in K^\circ}\langle x, y\rangle=h_{K^\circ}(x).
\end{align*}
It is clear that
\begin{align*}
  \|x\|_K=1 \,\,\,\,\,\,\,\, \mbox{whenever} \,\,\, \,\,x\in \partial K.
\end{align*}

Recall that the $L_p$ ($p\geq 1$) Minkowski combination of convex bodies $K$ and $L$ is defined as
\begin{align}
  h_{K+_p\epsilon \cdot L}(x)^p=h_K(x)^p+\epsilon h_L(x)^p.
\end{align}
One of the most important inequality related to  the $L_p$ Brunn-Minkowski combination of convex bodies $K$ and $L$ is
\begin{align*}
  V(K+_pL)^{\frac{p}{n}}\geq V(K)^{\frac{p}{n}}+V(L)^{\frac{p}{n}},
\end{align*}
with equality if and only if $K$ and $L$ are dilation of each other. The $L_p$ surface area  measure of $K$ is defined by
\begin{align}
  dS_p(K,\cdot)=h_K^{1-p}dS(K,\cdot),
\end{align}
where $dS(K,\cdot)$ is the classical surface area measure, which is given by
$$\lim_{\epsilon \rightarrow 0^+}\frac{V(K+\epsilon L)-V(K)}{\epsilon}=\int_{S^{n-1}}h_Q(u)dS(K,u).$$
It is easy to say that for $\lambda>0$, $S_p(\lambda K,\cdot)=\lambda^{n-p}S_p(K,\cdot)$. If $K\in \mathcal K^n_o$, then $K$ has a curvature function, then $f_p(K,\cdot):S^{n-1}\rightarrow \mathbb R$, the $L_p$-curvature function of $K$, is defined by
\begin{align*}
  f_p(K,\cdot)=h_K^{1-p}f(K,\cdot),
\end{align*}
where $f(K,\cdot)$ is the curvature function, $f(K,\cdot): S^{n-1}\rightarrow \mathbb R^n$ defined as the Radon-Nikodym derivative
$$f(K,\cdot)=\frac{dS(K,\cdot)}{dS},$$
and $dS$ is the standard Lebesgue measure on $S^{n-1}$.

For quick reference about the definition and notations in convex geometry, good references are Gardner \cite{gar-geo2006}, Gruber \cite{gru-con2007}, Schneider \cite{sch-con1993}.

\subsection{Functional setting}

In the following, we discuss in the functional setting in $\mathbb R^n$. Let $u:\mathbb R^n\rightarrow \mathbb R\cup\{+\infty\}$ be a convex function, that is $u\big((1-t)x+ty\big)\leq (1-t)u(x)+tu(y)$ for $t\in (0,1)$. We set
$dom(u)=\{x\in \mathbb R^n:u(x)\in \mathbb R\},$
by the convexity of $u$, $dom(u)$ is a convex set in $\mathbb R^n$. We say that $u$ is proper if $dom(u)\neq \emptyset$, and $u$ is of class $\mathcal C^2_+$ if it is twice differentiable on $int\big(dom(u)\big)$, with a positive definite Hessian matrix.
Recall that the Fenchel conjugate of $u$ is the convex function defined by
\begin{align}\label{leg-con}
  u^*(y)=\sup_{x\in\mathbb R^n}\big\{\langle x,y\rangle -u(x)\big\}.
\end{align}
It is obvious that $u(x)+u^*(y)\geq \langle x,y\rangle$ for $x, y\in \mathbb R^n$, there is an equality if and only if $x\in dom(u)$ and $y$ is in the subdifferential of $u$ at $x$, that means
\begin{align}\label{leg-equ}
  u^*(\nabla u(x))+u(x)=\langle x,\nabla u(x)\rangle.
\end{align}
The convex function $u:\mathbb R^n\rightarrow \mathbb R\cup\{+\infty\}$ is lower semi-continuous,  if the subset
$\{x\in \mathbb R^n: u(x)>t\}$ is an open set for any $t\in (-\infty,+\infty]$.
Moreover, if $u$ is a lower semi-continuous convex function, then also $u^*$ is a lower semi-continuous convex function, and $u^{**}=u$.


The infimal convolution of functions $u$ and $v$ from $\mathbb R^n$ to $\mathbb R\cup\{+\infty\}$ is defined by
\begin{align}\label{inf-con}
  u\Box v(x)=\inf_{y\in \mathbb R^n}\big\{u(x-y)+v(y)\big\}.
\end{align}
The right scalar multiplication by a nonnegative real number $\alpha$:
\begin{align}\label{rig-mul}
\big(u\alpha\big)(x):=\left\{
           \begin{array}{ll}
           \alpha u\big(\frac{x}{\alpha}\big), & \hbox{$if\,\, \alpha >0$;} \\
           I_{\{0\}}, & \hbox{$if\,\,\alpha=0$.}
           \end{array}
         \right.
\end{align}

The following results below gather some elementary properties of $u$, the Fenchel conjugate and the infimal convolution, which can be found in \cite{col-fra-the2013,roc-con1970}.
\begin{lem}\label{con-for-u}
  Let $u:\mathbb R^n\rightarrow \mathbb R\cup\{+\infty\}$, then there exist constants $a$ and $b$, with $a>0$, such that, for $\forall x\in \mathbb R^n$
\begin{align}
  u(x)\geq a\|x\|+b.
\end{align}
Moreover $u^*$ is proper, and satisfies $u^*(y)>-\infty$, $\forall y\in\mathbb R^n$.
\end{lem}

\begin{prop}
  Let $u:\mathbb R^n\rightarrow \mathbb R\cup\{+\infty\}$ be a convex function. Then:

\normalfont (1) $\big(u\Box v\big)^*=u^*+v^*$;

 \normalfont (2) $(u\alpha)^*(x)=\alpha u^*(\frac{x}{\alpha}), \,\,\,\,\,\alpha>0$;

\normalfont (3) $dom(u\Box v)=dom(u)+dom(v)$;

\normalfont (4) it holds $u^*(0)=-\inf(u)$, in particular if $u$ is proper, then $u^*(y)>-\infty$; $\inf(u)>-\infty$ implies $u^*$ is proper.
            \end{prop}


A function $f: \mathbb R^n\rightarrow \mathbb R\cup\{+\infty\}$ is called log-concave if for $x, y\in \mathbb R^n$ and $0<t<1$, we have
\begin{align*}
  f\big((1-t)x+ty \big)\geq f^{1-t}(x)f^t(y).
\end{align*}
If $f$ is a strictly positive log-concave function on $\mathbb R^n$, then there exist a convex function $u:\mathbb R^n\rightarrow \mathbb R\cup\{+\infty\}$ such that $f=e^{-u}$.
Following the notations in paper \cite{{col-fra-the2013}}, let
$$\mathcal L=\big\{u:\mathbb R^n\rightarrow \mathbb R^n\cup\{+\infty\}| \mbox{ proper,  convex,} \lim\limits_{|x|\rightarrow +\infty}u(x)=+\infty\}$$
$$\mathcal A=\big\{f:\mathbb R^n\rightarrow \mathbb R|\,\,f=e^{-u}, u\in \mathcal L\big\}.$$

Let $f\in \mathcal A$ be a log-concave, according to a series of papers by Artstein-Avidan and Milman \cite{avi-mil-a2010}, Rotem \cite{rot-on2012},  the support function of $f=e^{-u}$ is defined as,
\begin{align}
  h_f(x)=(-\log f(x))^*=u^*(x).
\end{align}
Here the $u^*$ is the Fenchel conjugate of $u$. The definition of $h_f$ is a proper generalization of the support function $h_K$, in fact, one can easily checks $h_{\chi_K}=h_K$.
Obviously, the support function $h_f$ shares the most of the important properties of  $h_K$.

The polar function of $f=e^{-u}$ is defined by
$f^\circ=e^{-u^*}.$ Specifically,
$$f^\circ(y)=\inf_{x\in \mathbb R^n}\Big\{\frac{e^{-\langle x,y\rangle}}{f(x)}\Big\},$$
it follows that, $f^\circ$ is also a log-concave function.

Let $\phi\in GL(n)$,  we always write $f\circ \phi(x)=\phi f(x)=f(\phi x)$.  The following proposition shows that $h_f$ is $GL(n)$ covariant which is proved in \cite{fan-zho-lyz2018}.
\begin{prop}\label{sln-invar}
  Let $f\in \mathcal A$. For $\phi\in GL(n)$ and $x\in \mathbb R^n$, then
  \begin{align*}
    h_{\phi f}(x)=h_f(\phi^{-t}x).
\end{align*}
Moreover, for the polar  function of $f$,
\begin{align*}
  (\phi f)^\circ=\phi^{-t}f^\circ.
     \end{align*}
\end{prop}

The class of log-concave functions $\mathcal A$ can be endowed with an algebraic structure which extends in a natural way as the usual of the Minkowski's structure on $\mathcal K^n$. For example, the Asplund sum of two log-concave functions is corresponded to the classical Minkowski sum of two convex bodies. See \cite{col-fra-the2013} for more about the Asplund sum and the related inequalities of the total mass of the log-concave function in $\mathcal A$ which correspond to the convex  bodies in $\mathcal K^n$. In very recently, the $L_p$ Asplund sum of log-concave functions are studied by author Fang, Xing and Ye \cite{fan-xin-ye-geo2020}. Let $$\mathcal A_0=\{e^{-u}:u\in\mathcal L_0\}\subset \mathcal A$$
 with
$$\mathcal L_0=\big\{u\in \mathcal L: u\geq 0, (u^*)^*=u  \mbox{ and }  u(o)=0\big \}.$$
Clearly, if $u\in\mathcal L_0$, then $u$ has its minimum attained at the origin $o$.
\begin{defn}[\cite{fan-xin-ye-geo2020}]\label{def-lp-asp-sum}
Let $f=e^{-u}$, $g=e^{-v}\in \mathcal A_0$, and $\alpha, \beta\geq 0$. The $L_p$ ($p\geq 1$) Asplund sum and multiplication of $f$ and $g$ is defined as
\begin{align}
  \alpha\cdot_p f\oplus_p \beta\cdot_p g=e^{-[(u\cdot_p\alpha)\Box_p(v\cdot_p\beta)]},
\end{align}
where
\begin{align*}
  (u\cdot_p\alpha)\Box_p(v\cdot_p\beta)=\Big[\big(\alpha (u^*)^p+\beta(v^*)^p\big)^{\frac{1}{p}}\Big]^*.
\end{align*}
\end{defn}
The $L_p$ Asplund sum is an extension of the Asplund sum on $\mathcal A_0$. Specially, when $p=1$, it reduces to the  Asplund sum of two functions on $\mathcal A_0$, that is
\begin{align}\label{f-plu-g}
\big(\alpha\cdot f\oplus \beta\cdot g\big)(x)=\sup_{y\in\mathbb R^n}f\Big(\frac{x-y}{\alpha}\Big)^\alpha g\Big(\frac{y}{\beta}\Big)^\beta.
\end{align}
Moreover, when $\alpha=0$ and $\beta>0$, we have $(\alpha\cdot_p f\oplus_p\beta\cdot_p g)(x)=g\big(\frac{x}{\beta^{\frac{1}{p}}}\big)^{\beta^{\frac{1}{p}}}$; when $\alpha >0$ and $\beta=0$, then $(\alpha\cdot_p f\oplus \beta\cdot_p g)(x)=f\big(\frac{x}{\alpha^{\frac{1}{p}}}\big)^{\alpha^{\frac{1}{p}}}$;  finally, when $\alpha=\beta=0$, we set  $\big(\alpha\cdot_p f\oplus_p \beta\cdot_p g\big)=I_{\{0\}}$. We say that the $L_p$ Asplund sum for log-concave functions is closely related to the $L_p$ Minkowski sum for convex bodies in $\mathbb R^n$. For examples, $K, \,\,L\in\mathcal K^n$, let
\begin{align}\label{chi-fun}
  \chi_K(x)=e^{-I_K(x)}=\left\{
                                 \begin{array}{ll}
                                   1, & \hbox{if \,\, $x\in K$;} \\
                                   0, & \hbox{if \,\, $x\notin K$,}
                                 \end{array}
                               \right.
\end{align}
where $I_K$ is the indicator function of $K$, and it is a lower semi-continuous convex function,
\begin{align}\label{ind-fun}
I_K(x)=\left\{
               \begin{array}{ll}
                 0, & \hbox{if \,\, $x\in K$;} \\
                 \infty, & \hbox{if \,\, $x\notin K$.}
               \end{array}
             \right.
\end{align}
The characteristic function $\chi_K$ is log-concave functions with $u=I_K$ belongs to $\mathcal L$, $u^*=h_K$ belongs to $\mathcal L$ if $0\in int(K)$, for  $p\in [1,+\infty)$, the function
\begin{align*}
  \big((I_K)\cdot_p\alpha\big)\Box_p\big((I_L)\cdot_p\beta\big)
  &=\Big[\big(\alpha(I_K^*)^p+\beta (I_L^*)^p\big)^{\frac{1}{p}}\Big]^*\\
  &=\Big[\big(\alpha h_K^p+\beta h_L^p\big)^{\frac{1}{p}}\Big]^*\\
  &=I_{\alpha\cdot_pK+_p\beta\cdot_p L}.
\end{align*}
Then  $\alpha\cdot \chi_K\oplus \beta\cdot \chi_L=e^{-[I_K\cdot_p\alpha\Box_p I_L\cdot_p\beta ]}=\chi_{\alpha\cdot_p K+_p\beta\cdot_pL}.$

The following Proposition assert that the $L_p$ Asplund sum of log-concave functions is closed in $\mathcal A_0$.

\begin{prop}[\cite{fan-xin-ye-geo2020}]\label{clo-uv}
Let $f$ and $g$ belong both to the same class $\mathcal A_0$, and $\alpha,\,\,\beta\geq 0$. Then $f\cdot_p\alpha \oplus_p \beta\cdot_p g$  belongs to  $\mathcal A_0$.
\end{prop}
The total mass function of $f$ is defined as
\begin{align}
  J(f)=\int_{\mathbb R^n}f(x)dx.
\end{align}
Clearly, when $f=\chi_K$, $J(f)=V(K)$.
Similar to the integral expression of mixed volume $V(K,L)$, for $f=e^{-u}$ and $g=e^{-v}$ in $\mathcal A_0$, the quantity $\delta J (f,g)$, which is called as the first variation of $J$ at $f$ along $g$ is defined by (see \cite{col-fra-the2013})
\begin{align*}
\delta J(f,g)=\lim_{t\rightarrow 0^+}\frac{J(f\oplus t\cdot g)-J(f)}{t}.
\end{align*}
It has been shown that $ \delta J(f,g)$ has the following integral expression,
\begin{align}\label{mix-que-exp}
\delta J(f,g)=\int_{\mathbb R^n}h_gd\mu(f,x),
\end{align}
where $\mu(f,x)$ is the surface area measure of $f$ on $\mathbb R^n$ and is given by
\begin{align}
  u(f,x)=(\nabla u(x))_{\sharp}f(\mathcal H^n),
\end{align}
here $\nabla u$ is the gradient of $u$ in $\mathbb R^n$, that means, for any Borel function $g\in \mathcal A$,
\begin{align}
  \int_{\mathbb R^n}g(x)d\mu(f,x)=\int_{\mathbb R^n}g(\nabla u(x))e^{-u(x)}dx.
\end{align}
Specially, if take $f=g$ in (\ref{mix-que-exp}), then
\begin{align}
  \delta J(f,f)=nJ(f)+\int_{\mathbb R^n}f\log fdx=J(nf+f\log f).
\end{align}
In the following sections we write $J(nf+f\log f)$ in terms of $J(f^\diamond)$ for simplicity.

The $L_p$ surface area measure of $f$, denoted by $\mu_p(f,\cdot)$ is given in \cite{fan-xin-ye-geo2020}.

\begin{defn}[\cite{fan-xin-ye-geo2020}]
Let $f=e^{-u}\in\mathcal A_0$ be a log-concave function, the $L_p$ surface area measure of $f$, denoted as $\mu_p(f,\cdot)$, is the Boreal measure on $\Omega$ such that
\begin{align}
  \int_{\Omega}g(y)d\mu_p(f,y)=\int_{\{x\in dom(u): \nabla u(x)\in\Omega\}}g(\nabla u(x))(h_f(\nabla u(x)))^{1-p}f(x)dx,
\end{align}
holds for every Borel function $g$ such that $g\in L^1\big(\mu_p(f,\cdot)\big)$.
\end{defn}
Similarly, the first variation of the total mass at $f$ along $g$ with respect to the $L_p$ Asplund sum is defined as,
\begin{defn}[\cite{fan-xin-ye-geo2020}]\label{def-p-mix}
  Let $f,\,g\in\mathcal A_0$. For $p\geq 1$, the first variation of the total mass of $f$ along $g$ with respect to the $L_p$ Asplund sum is defined by
  \begin{align}
    \delta J_p(f,g)=\lim_{t\rightarrow 0^+}\frac{J(f\oplus_p t\cdot_pg)-J(f)}{t},
  \end{align}
  whenever the limit exists.
\end{defn}
The following integral expression of $\delta J_p(f,g)$ is $L_p$ extension of (\ref{mix-que-exp}) which is  established in \cite{fan-xin-ye-geo2020}. Specially, if take $f=e^{-I_K(x)}$ and $g=e^{-I_L(x)}$, where $I_K$ and $I_L$ are the indicator function of $K$ and $L$. So $J(f)=V(K)$ and $J(g)=V(L)$, then we have $\delta J(f,g)=V_p(K,L)$.

\begin{thm}[\cite{fan-xin-ye-geo2020}]
  Let $f=e^{-u}\in \mathcal A_0$ and $g=e^{-v}\in\mathcal A_0$. For $p\geq 1$, assume that $g$ is an admissible $p$-perturbation for $f$. In addition, suppose that there exists a constant $k>0$ such that
  \begin{align*}
    det\big(\nabla^2(h_f)^p\big)\leq k(h_f)^{n(p-1)}det(\nabla^2h_f),
  \end{align*}
  holds for all $x\in\mathbb R^n\backslash\{o\}$. Then
  \begin{align}\label{fir-var-jf}
    \delta J_p(f,g)=\frac{1}{p}\int_{\mathbb R^n}(h_g)^pd\mu_p(f,x).
  \end{align}
\end{thm}
Note that  the support function of  the log-concave function is nondecreasing, it's easy to get  that if $g_1\leq g_2$, then $$\delta J_p(f,g_1)\leq \delta J_p(f,g_2).$$

In the following, we normalize the  $\delta J_p(f,g)$. For $f=e^{-u},\,\,g=e^{-v}\in \mathcal A_0$, and $1\leq p<\infty$, we define
\begin{align}\label{nor-det}
  \overline\delta J_p(f,g)&=\Big(\frac{p\cdot\delta J_p(f,g)}{J(f^\diamond)}\Big)^{\frac{1}{p}}
 =\Big[\frac{1}{J(f^\diamond)}\int_{\mathbb R^n}\Big(\frac{h_g}{h_f}\Big)^ph_fd\mu(f,x)\Big]^{\frac{1}{p}},
\end{align}
Note that $\frac{h_fd\mu(f,x)}{J(f^\diamond)}$ is a probability measure on $\mathbb R^n$.
For $p=\infty$ define
\begin{align}\label{inf-delta-p}
  \overline\delta J_\infty(f,g)=\max\Big\{\frac{h_g(x)}{h_f(x)}: x\in\mathbb R^n\Big\}.
\end{align}
Unless $\frac{h_g(x)}{h_f(x)}$ is a constant on $\mathbb R^n$, by the Jensen's inequality,  it follows that $\overline\delta J_p(f,g)< \overline\delta J_q(f,g)$, for $1\leq p<q<\infty$. For $p=\infty$, we have $\lim_{p\rightarrow \infty}\overline\delta J_p(f,g)=\overline\delta J_\infty(f,g)$. Moreover, we have the following Lemma.
\begin{lem}
  Suppose $f=e^{-u},\,g=e^{-v}\in \mathcal A_0$, $1\leq p<q<\infty$. Then
\begin{align}
  \overline\delta J_1(f,g)\leq \overline\delta J_p(f,g)\leq \overline\delta J_q(f,g)\leq\overline\delta J_\infty(f,g).
\end{align}
\end{lem}

In order to establish the continuity of the $L_p$ John ellipsoid for log-concave function in section 4, we give the following Lemma of the $\overline\delta J_p(f,g)$.
\begin{lem}
  Let $f=e^{-u},\,\,g=e^{-v},\,\,g_0=e^{-v_0}\in \mathcal A_0$, then
\begin{align}\label{lip-ineq}
  \big|\overline\delta J_p(f,g)-\overline\delta J_p(f,g_0)\big|\leq \frac{\|h_g-h_{g_0}\|_\infty}{\min \{ |h_f|:x\in\mathbb R^n\}},
\end{align}
for all $p\in [1,\infty]$, where  $\|\cdot\|_\infty$ denotes the $\infty$ norms.
\end{lem}
\begin{proof}
  First suppose that $p<\infty$, by (\ref{nor-det}) and the triangle inequality for $L_p$ norms, we have
\begin{align*}
\big|\overline\delta J_p(f,g)-\overline\delta J_p(f,g_0)\big|&\leq\Big[\frac{1}{J(f^\diamond)}\int_{\mathbb R^n}\Big|\frac{h_g}{h_f}-\frac{h_{g_0}}{h_f}\Big|^ph_fd\mu(f,x)\Big]^{\frac{1}{p}}\\
&\leq\Big[\frac{1}{J(f^\diamond)}\int_{\mathbb R^n}\frac{1}{|h_f|^p}h_fd\mu(f,x)\Big]^{\frac{1}{p}}\|h_g-h_{g_0}\|_\infty\\
&\leq \frac{\|h_g-h_{g_0}\|_\infty}{\min\{|h_f(x)|:x\in \mathbb R^n\}}.
\end{align*}
The third inequality we use the fact that $\frac{h_fd\mu(f,x)}{J(f^\diamond)}$ is a probability measure on $\mathbb R^n$.
For $p\rightarrow\infty$, the continuous of the $L_p$ norm with $p$ shows that  (\ref{lip-ineq}) holds for $p=\infty$ as well.
\end{proof}

The following Lemma shows some Properties of $\delta J_p(f,g)$ and its normalizer.
\begin{lem}\label{inv-det-j}
  Suppose that $f=e^{-u},\,\, g=e^{-v}\in \mathcal A_0$, then

  \normalfont  (1) $\delta J_p(f,f)=\frac{1}{p}J(f^\diamond)$.

  \normalfont  (2) $\overline \delta J_p(f, f)=1$.

  \normalfont (3) $\delta J_p(f,\lambda \cdot_p g)=\lambda\delta J_p(f,g)$, for $\lambda>0$.

  \normalfont (4) $\overline\delta J_p(f,\lambda \cdot_p g)=\lambda^{\frac{1}{p}}\overline\delta J_p(f,g)$, for $\lambda>0$.

  \normalfont (5) $\delta J_p(\phi f,g)=|\det\phi|^{-1}\delta J_p(f,\phi^{-1}g)$, for all $\phi\in GL(n)$.

  \normalfont (6) $\overline\delta J_p(\phi f,g)=\overline\delta J_p(f,\phi^{-1}g)$, for all $\phi\in GL(n)$.
\end{lem}
\begin{proof}
  By formula (\ref{fir-var-jf}) and Definition \ref{def-p-mix}, it immediately gives (1) and (2).

  In order to prove (3), by Definition \ref{def-lp-asp-sum}, we have
 $h_{\lambda\cdot_p f}=\lambda^{\frac{1}{p}} h_f.$
  So we have $ \delta J_p(f,\lambda\cdot_p g)=\lambda\delta J_p(f,g).$
 By (3), it yields (4) directly.

 By the integral formula of the first variation (\ref{fir-var-jf}), and note that $\nabla_x(\phi u )=\phi^{t}\nabla_{\phi x}u $, we have
  \begin{align*}
    \delta J_p(\phi f, g)&=\frac{1}{p}\int_{\mathbb R^n}h^p_g(x)d\mu_p(\phi f,x)\\
    &=\frac{1}{p}\int_{\mathbb R^n}h_g^p\big(\nabla_x(\phi u)\big)h^{1-p}_{\phi f}\big(\nabla_x(\phi u)\big)e^{- \phi u}dx\\
    &=\frac{1}{p}\int_{\mathbb R^n}h^p_{\phi^{-1}g }\big(\nabla_{\phi x}u \big)h^{1-p}_f\big(\nabla_{\phi x}u \big)e^{- u(\phi x)}dx\\
    &=|\det\phi^{-1}|\frac{1}{p}\int_{\mathbb R^n}h^p_{\phi^{-1}g}\big(\nabla u \big)h^{1-p}_f\big(\nabla u \big)e^{- u }dx\\
    &=|\det\phi|^{-1}\delta J_p(f, \phi^{-1}g).
  \end{align*}

On other hand, note that $J\big((\phi f)^\diamond\big)=|\det \phi|^{-1}J(f^\diamond)$ and together with (5), it follows that  $\overline\delta J_p(\phi f,g)=\overline\delta J_p(f,\phi^{-1}g)$. So we complete the proof.
\end{proof}

The following result will been used in the next section (see \cite{rot-sup2013}).
\begin{prop}[\cite{rot-sup2013}]\label{con-the}
  Let $D$ be a relatively open convex sets, and $f_1, f_2,\cdots, $ be a sequence of finite convex functions on $D$. Suppose that the real number $f_1(x), f_2(x),\cdots, $ is bounded for each $x\in D$. It is then possible to selected a subsequence of $f_1,f_2,\cdots$, which converges uniformly on closed bounded subset of $D$ to some finite convex function $f$.
\end{prop}

\section{$L_p$ John ellipsoid for log-concave functions}

Let $\gamma=e^{-\frac{\|x\|^2}{2}}$ be the standard Gaussian function. In the following, we set
$$\gamma_\phi(x)=e^{-\frac{\|\phi x\|^2}{2}},$$
where $\phi\in GL(n)$.
It is worth noting that the Gaussian function $\gamma_\phi$ plays an important role in the study of the extremal problem of log-concave functions as the ellipsoids do for the study of the extremal problems of convex bodies. In fact, it is the unique function of  $\mathcal A$ which is self-dual, that is
$$f=e^{-\frac{\|x\|^2}{2}}\Longleftrightarrow f^o=f.$$

Now, Let us consider the following optimization problem.


The $L_p$ optimization problem $S_{p}$ ($p\geq 1$) for log-concave functions $f$: Given a log-concave function $f\in\mathcal A_0,$ find a Gaussian function $\gamma_\phi $ which solves the following constrained maximization problem:
\begin{align}
  \max\Big(\frac{J(\gamma_\phi)}{c_n}\Big)\,\,\,\,\mbox{ subject to}\,\,\,\,\,\,   \overline\delta J_p\big(f,\gamma_\phi\big)\leq 1.
\end{align}

The dual $L_p$ optimization problem $\overline S_p$ ($p\geq 1$) for log-concave functions $f$: Given a log-concave function $f\in\mathcal A_0$, find a Gaussian function $\gamma_{\overline\phi}$ which solves the following constrained minimization problem:
\begin{align}
  \min\overline \delta J_p\big(f,\gamma_{\overline\phi}\big)\,\,\,\,\mbox{ subject to}\,\,\,\,\,\,\frac{ J(\gamma_{\overline\phi})}{c_n}\geq 1.
\end{align}

\begin{lem}\label{solut-for S}
These above optimization problems for log-cocave functional are equivalent to:
{\normalfont

(1)} The problem $S_p$ is equivalent to:
$$\max\Big(\frac{J(\gamma_\phi)}{c_n}\Big)\,\,\,\,\,\,\,\, \mbox{ subject to} \,\,\,\,\,\,\,\overline \delta J_p(f, \gamma_\phi)=1.$$

{\normalfont (2)} The dual problem $\overline S_p$ is equivalent to:
$$\min\overline \delta J_p(f,\gamma_{\overline\phi})\,\,\,\,\,\,\,\, \mbox{ subject to} \,\,\,\,\,\,\, \frac{J(\gamma_{\overline\phi})}{c_n}=1.$$
\end{lem}
\begin{proof}
(1) Assume that  $\gamma_\phi=e^{-\frac{\|\phi x\|^2}{2}}$ is the solution of the problem $S_p$, if  $\frac{J(\gamma_\phi)}{c_n}$ obtains the maximum, but $\overline \delta J_p(f,\gamma_\phi)\neq 1$, assume that $\overline \delta J_p(f,\gamma_\phi)< 1$, then let $$\overline\gamma_\phi=\frac{1}{\overline \delta J_p(f,\gamma_\phi)^p}\cdot_p\gamma_{\phi}.$$
Then $h_{\overline\gamma_\phi}=\frac{1}{\overline \delta J_p(f,\gamma_\phi)}h_{\gamma_{\phi}}$. Moreover, we have
\begin{align*}
\overline \delta J_p(f, \overline\gamma_\phi)&=\Big[\frac{1}{J(f^\diamond)}\int_{\mathbb R^n}h_{\overline\gamma_\phi}^p d\mu_p(f,x)\Big]^{\frac{1}{p}}\\
&=\frac{1}{\overline \delta J_p(f,\gamma_\phi)}\Big[\frac{1}{J(f^\diamond)}\int_{\mathbb R^n}h_{\gamma_\phi}^pd\mu_p(f,x)\Big]^{\frac{1}{p}}\\
&=1.
\end{align*}
On the other hand, since $\gamma_\phi=e^{-\frac{\|\phi x\|^2}{2}}$, by the Definition \ref{def-lp-asp-sum} of the $L_p$ multiplication, a simple computation shows that, for $\lambda\in \mathbb R$,
\begin{align}\label{mult-int-equ}
  J\Big(\lambda^p\cdot_p\gamma_\phi\Big)=\int_{\mathbb R^n}e^{-\lambda \frac{\|\phi (x/\lambda)\|^2}{2}}dx=
\lambda^{\frac{n}{2}}J(\gamma_\phi).
\end{align}
Note that $\overline\delta J_p(f,\gamma_\phi)^{-\frac{n}{2}}>1$, then we have
\begin{align*}
  J\Big(\frac{1}{\overline \delta J_p(f,\gamma_\phi)^p}\cdot_p\gamma_\phi\Big)=\overline \delta J_p(f,\gamma_\phi)^{-\frac{n}{2}}J(\gamma_\phi)\geq J(\gamma_\phi).
\end{align*}
This means that $\gamma_\phi$ is not a solution to the problem $S_p$, which is contract with our assumption, so  we complete the proof of the first one.
 The similar with the proof of Problem $\overline {S_p}$, so we complete the proof.
\end{proof}

\begin{lem}\label{ove-S to S}
  Suppose $f\in \mathcal A_0$. If $\gamma_\phi$ is a Gaussian function that is an solution for the problem $S_p$ of $f$, then
  $$\Big(\frac{c_n}{J(\gamma_\phi)}\Big)^{\frac{2p}{n}}\cdot_p\gamma_\phi$$ is an solution  for problem $\overline S_p$  of $f$. Conversely, if $\gamma_{\overline\phi}$ is a Gaussian function that is an solution for the $\overline S_p$ problem of $f$, then
  $$\frac{1}{\overline\delta J_p(f,\gamma_{\overline\phi})^p}\cdot_p\gamma_{\overline\phi}$$
  is an solution  for problem $S_p$ of $f$.
\end{lem}
\begin{proof}
Let $\gamma_T=e^{-\frac{\|T x\|}{2}}$, where $T\in GL(n)$, and satisfies $J(\gamma_T)\geq c_n$.  By Lemma \ref{inv-det-j}, it obviously has
$\overline\delta J_p\Big(f, \frac{1}{\overline \delta J_p(f,\gamma_T)^p}\cdot_p\gamma_T\Big)=1.$
Since $\gamma_\phi$ is an $S_p$ solution for $f$, then
$$J(\gamma_\phi)\geq J\Big(\frac{1}{\overline \delta J_p(f,\gamma_T)^p}\cdot_p\gamma_T\Big).$$
By (\ref{mult-int-equ}) it  shows
\begin{align}\label{p-lin-int-log}
J\Big(\frac{1}{\overline \delta J_p(f,\gamma_T)^p}\cdot_p\gamma_T\Big)=\overline \delta J_p(f,\gamma_T)^{-\frac{n}{2}}J(\gamma_T).
\end{align}
According to Lemma \ref{solut-for S}, we have $\overline \delta J_p(f,\gamma_\phi)=1$. Then
$$\overline \delta J_p(f,\gamma_T)\geq \Big(\frac{J(\gamma_T)}{J(\gamma_\phi)}\Big)^{\frac{2}{n}}\geq \Big(\frac{c_n}{J(\gamma_\phi)}\Big)^{\frac{2}{n}}=\overline \delta J_p\Big(f,\big(\frac{c_n}{J(\gamma_\phi)}\big)^{\frac{2p}{n}}\cdot_p\gamma_\phi\Big).$$
On the other hand, since
$J\Big(\big(\frac{c_n}{J(\gamma_\phi)}\big)^{\frac{2p}{n}}\cdot_p\gamma_\phi\Big)=c_n,$
it implies that the ellipsoid $\big(\frac{c_n}{J(\gamma_\phi)}\big)^{\frac{2p}{n}}\cdot_p\gamma_\phi$ is a solution to Problem $\overline S_p$. We complete the first assertion's proof.

Let $\gamma_{\overline T}=e^{-\frac{\|\overline Tx\|^2}{2}}$, $\overline T\in GL(n)$, such that $\overline \delta J_p(f,\gamma_{\overline T})\leq 1$. Then we have
$$J\Big(\big(\frac{c_n}{J(\gamma_{\overline T})}\big)^{\frac{2p}{n}}\cdot_p\gamma_{\overline T}\Big)\geq c_n.$$
Since $\gamma_{\overline \phi}$ is solution of $\overline S_p$, we have
$$\Big(\frac{c_n}{J(\gamma_{\overline T})}\Big)^{\frac{2}{n}}\overline \delta J_p(f,\gamma_{\overline T})=\overline \delta J_p\Big(f,\big(\frac{c_n}{J(\gamma_{\overline T})}\big)^{\frac{2p}{n}}\cdot_p\gamma_{\overline T}\Big)\geq \overline \delta J_p(f,\gamma_{\overline \phi}).$$
By Lemma \ref{solut-for S} we have $J(\gamma_{\overline \phi})=c_n$. Hence the above inequality can be rewritten as
$$\frac{J(\gamma_{\overline\phi})}{\overline\delta J_p(f,\gamma_{\overline\phi})^{\frac{n}{2}}}\geq \frac{J(\gamma_{\overline T})}{\overline\delta J_p(f,\gamma_{\overline T})^{\frac{n}{2}}}.$$
By formula (\ref{p-lin-int-log}) again, it means
\begin{align*}
  J\Big(\frac{1}{\overline\delta J_p(f,\gamma_{\overline\phi})^p}\cdot_p\gamma_{\overline\phi}\Big)\geq  J\Big(\frac{1}{\overline\delta J_p(f,\gamma_{\overline T})^p}\cdot_p\gamma_{\overline T}\Big)\geq J(\gamma_{\overline T}).
\end{align*}
On the other hand, since
\begin{align*}
  \overline \delta J_p\Big(f,\frac{1}{\overline\delta J_p(f,\gamma_{\overline\phi})^p}\cdot_p\gamma_{\overline\phi}\Big)=1.
\end{align*}
This completes the proof.
\end{proof}

\begin{thm}\label{exi-sol-sp}
Suppose that $f\in\mathcal A_0$. Then there exist a solution, $\gamma_{\overline\phi}$, satisfies the Optimization Problem $\overline S_p$.
  \end{thm}
\begin{proof}
  By the definition of the optimization problem $\overline S_p$, let $\phi\in GL(n)$, if $\gamma_\phi$ is Gaussian function subject to  $J(\gamma_\phi)=c_n$, then we have $\phi\in SL(n)$. The existence of the solution of the Problem $\overline S_p$ is equivalent to find the $\phi_0\in SL(n)$ which solves the following. Let $\phi\in SL(n)$, choose $\epsilon_0>0$ sufficiently small so that for all $\epsilon\in (-\epsilon_0, \epsilon_0)$, $I+\epsilon \phi$ is invertible. For $\epsilon\in (-\epsilon_0, \epsilon_0)$ define $\phi_\epsilon\in SL(n)$ by
\begin{align*}
  \phi_\epsilon=\frac{I+\epsilon \phi}{det(I+\epsilon \phi)^{\frac{1}{n}}},
\end{align*}
here $I$ is the identity matrix and $|det\phi_\epsilon|=1$. Then find $\phi_0$ such that
\begin{align}
  \frac{d}{d\epsilon}\Big|_{\epsilon=0}\delta J_p\big(f,\gamma_{\phi_\epsilon} \big)=0.
\end{align}
It is equivalent to
\begin{align*}
\int_{\mathbb R^n}\frac{d}{d\epsilon}\big|_{\epsilon=0}\Big(\frac{\|\phi^{-t}_\epsilon x\|^2}{2}\Big)^{p}d\mu_p(f,x)=0.
\end{align*}
Since the norm of the vector and $t^p$ are continuous functions, it grants there exists a solution of the above equation. So there exists a $\phi\in SL(n)$ such that $\overline\delta J_p(f,\gamma_\phi)$.
\end{proof}

We say that the $\gamma_{\phi}$ is the solution of the $L_p$ functional optimization problem $S_p$, and we rewrite it as the $\gamma_ f$. The following Corollary are obviously.
\begin{cor}\label{mix-aff-que}
Suppose $f\in \mathcal A_0$, $\gamma_f$ be the solution of the optimization problem $S_p$, then
\begin{align}
 \overline \delta J_p(f, \gamma_f)=1.
\end{align}
\end{cor}

By Lemma \ref{ove-S to S} and Theorem \ref{exi-sol-sp}, it guarantees that there exists a unique solution for the optimization problem $\overline S_p$.
\begin{thm}\label{suf-sol-ove-sp}
  Let $f\in\mathcal A_0$. Then problem $\overline S_p$ has a unique solution. Moreover a Gaussian function $\gamma_{\overline\phi}$ solves $\overline S_p$  if and only if it satisfies
  \begin{align}\label{equi-solut-oversp}
\delta J_p(f,\gamma_{\overline\phi})h_{\gamma_{\overline\phi}^o}(y)=\frac{n}{4p}\int_{\mathbb R^n}|\langle x,y\rangle|^2 h_{\gamma_{\overline\phi}^o}^{p-1}(x)d\mu_p(f,x),
\end{align}
for any $y\in \mathbb R^n$.
\end{thm}
In order to prove Theorem \ref{suf-sol-ove-sp}, we need the following  Lemma.

\begin{lem}\label{uni-sol-lp-john}
Let $f=e^{-u}\in \mathcal A_0$ and $\overline \phi\in GL(n)$. If the Gaussian function $\gamma_{\overline\phi}$ solves the optimization functional problem $\overline S_p$, then
\begin{align*}
\delta J_p(f,\gamma_{\overline\phi})h_{\gamma_{\overline\phi}^o}(y)=\frac{n}{4p}\int_{\mathbb R^n}|\langle x,y\rangle|^2 h_{\gamma_{\overline\phi}^o}^{p-1}(x)d\mu_p(f,x).
\end{align*}
\end{lem}
\begin{proof}
 By the $SL(n)$ invariance of the $\delta J_p(f,g)$, we may assume that $\gamma_{\overline\phi}=\gamma$ is the solution of problem $\overline S_p$. Let $\phi\in SL(n)$, choose $\epsilon_0>0$ sufficiently small so that for all $\epsilon\in (-\epsilon_0, \epsilon_0)$, $I+\epsilon \phi$ is invertible. For $\epsilon\in (-\epsilon_0, \epsilon_0)$ define $\phi_\epsilon\in SL(n)$ by
\begin{align*}
  \phi_\epsilon=\frac{I+\epsilon \phi}{det(I+\epsilon \phi)^{\frac{1}{n}}},
\end{align*}
here $I$ is the identity matrix and $|det\phi_\epsilon|=1$. Then
$$\delta J_p\big(f,\gamma_{\phi_0}\big)\leq\delta J_p\big(f,\gamma_{\phi_\epsilon} \big),$$
 for all $\epsilon \in(-\epsilon_0, \epsilon_0)$. That means
\begin{align}
  \frac{d}{d\epsilon}\Big|_{\epsilon=0}\delta J_p\big(f,\gamma_{\phi_\epsilon} \big)=0.
\end{align}
On the other hand, by Proposition \ref{sln-invar}, we have $\gamma_{\phi_\epsilon}^o=\phi_\epsilon^{-t} \gamma$. Then $h_{\gamma_{\phi_\epsilon}^o}(x)=h_{\phi_\epsilon^{-t}\gamma}(x)=h_{\gamma}(\phi_\epsilon x).$
By the definition of $\delta J_p(f,g)$, we have
\begin{align*}
\frac{d}{d\epsilon}\Big|_{\epsilon=0}\int_{\mathbb R^n}h^{p}_{\gamma_{\overline\phi}^o}(x) d\mu_p(f,x)=0.
\end{align*}
It is equivalent to the following
\begin{align*}
\int_{\mathbb R^n}\frac{d}{d\epsilon}\big|_{\epsilon=0}\Big(\frac{\|\phi_\epsilon x\|^2}{2}\Big)^{p}d\mu_p(f,x)=0.
\end{align*}
Or equivalently,
\begin{align*}
\int_{\mathbb R^n}\frac{d}{d\epsilon}\big|_{\epsilon=0}&\Big(det(I+\epsilon\phi)^{-\frac{2p}{n}}\big(\langle x\cdot x\rangle+2\epsilon\langle x\cdot\phi x\rangle+\epsilon^2\langle \phi x\cdot\phi x\rangle \big)^{p}\Big)d\mu_p(f,x)=0.
\end{align*}
Since $\frac{d}{d\epsilon}|_{\epsilon=0}\det(I+\epsilon \phi)=trace(\phi)$, by a simple computation, we have,
\begin{align*}
trace(\phi)\delta J_p(f,\gamma)=\frac{n}{2p}\int_{\mathbb R^n}\langle x,\phi x\rangle\Big(\frac{\|x\|^2}{2}\Big)^{p-1}d\mu_p(f,x).
\end{align*}
Choosing an appropriate $\phi$ for each $i,\,j\in \{1,2,\cdots,n\}$ gives
\begin{align*}
\delta_{i,j}\delta J_p(f,\gamma)=\frac{n}{2p}\int_{\mathbb R^n}\langle x,e_i\rangle\langle x,e_j\rangle  h_\gamma^{p-1}(x)d\mu_p(f,x),
\end{align*}
where $e_1,\cdots, e_n$  is an otrhonormal basis of $\mathbb R^n$, and $\delta_{i,j}$ is the Koronecker symbols. Which in turn gives
\begin{align}\label{nor-delt-equ}
\|y\|^2\delta J_p(f,\gamma)=\frac{n}{2p}\int_{\mathbb R^n}|\langle x,y\rangle|^2 h_\gamma^{p-1}(x)d\mu_p(f,x).
\end{align}
That means
\begin{align}\label{suf-ine}
\delta J_p(f,\gamma)h_{\gamma}(y)=\frac{n}{4p}\int_{\mathbb R^n}|\langle x,y\rangle|^2 h_{\gamma}^{p-1}(x)d\mu_p(f,x).
\end{align}
So we have
\begin{align*}
\delta J_p(f,\gamma_{\overline\phi})h_{\gamma_{\overline\phi}^o}(y)=\frac{n}{4p}\int_{\mathbb R^n}|\langle x,y\rangle|^2 h_{\gamma_{\overline\phi}^o}^{p-1}(x)d\mu_p(f,x).
\end{align*}
We complete the proof.
\end{proof}

Now we prove Theorem \ref{suf-sol-ove-sp}.

\begin{proof}[Proof of Theorem \ref{suf-sol-ove-sp}]
Lemma \ref{uni-sol-lp-john} grants that if $\gamma_{\overline\phi}$ is and $\overline S_p$ solution of $f$, then the above formula holds.

Covnersely, without loss of generality, we may prove that (\ref{equi-solut-oversp}) holds when $\gamma_{\overline\phi}=\gamma$, that is $\overline\phi=I$. Then for any $\phi_1\in GL(n)$, we shall prove that if $J(\gamma_{\phi_1})=c_n$ for some $\phi_1\in GL(n)$,
\begin{align*}
  \delta J_p(f,\gamma_{\phi_1})\geq \delta J_p(f,\gamma),
\end{align*}
with equality if and only if $\gamma_{\phi_1} =\gamma$. Equivalently, we shall prove that if $\phi_1$ is a positive definite symmetric matric with $|\phi_1|=1$, and note that $\gamma_{\phi_1}=e^{-\frac{\|\phi_1 x\|}{2}}$ then
\begin{align*}
  \frac{1}{p \delta J_p(f,\gamma)}\int_{\mathbb R^n}h^p_{ \gamma_{\phi_1} }(x)d\mu_p(f,x)\geq 1,
\end{align*}
or equivanently,
\begin{align*}
  \frac{1}{ p\delta J_p(f,\gamma)}\int_{\mathbb R^n}\Big(\frac{\|\phi_1^{-t}x\|^2}{\|x\|^2}\Big)^{p}\Big(\frac{\|x\|^2}{2}\Big)^pd\mu_p(f,x)\geq 1,
\end{align*}
with equality if and only $\frac{\|\phi^{-t}_1x\|}{\|x\|}=1$, for $x\in \mathbb R^n$.

Write $\phi_1^{-t}=O TO^t$, where $T=diag(\lambda_1,\cdots,\lambda_n)$ is a diagonal matrix with eigenvalues $\lambda_1,\cdots,\lambda_n$ and $O$ is an orthogonal matrix. To establish our inequality we need to show that if $\phi_1$ is a positive definite symmetric matrix with $\det\phi_1=1$, then
$$\int_{\mathbb R^n}\log\Big(\frac{\|\phi_1^{-t} x\|}{\|x\|}\Big)\Big(\frac{\|x\|}{2}\Big)^pd\mu_p(f,x)\geq 0.$$
On the other hand, since $\nabla(\phi u )=\phi^t\nabla_{\phi x}u$ for each $\phi\in GL(n)$, and $(\phi u )^*=\phi^{-t}u^*$, we have
\begin{align}
 \nonumber \int_{\mathbb R^n}&\log\Big(\frac{\|T x \|^2}{\|x\|^2}\Big)^{p}\Big(\frac{\|x\|^{2}}
  {2}\Big)^{p}d\mu_p(O f,x)\\
 \nonumber  &=2p\int_{\mathbb R^n}\log\frac{\|TO^t\nabla_{O x} u\|}{\|O^t\nabla_{Ox}u\|}\Big(\frac{\|O^t\nabla_{Ox}u\|^{2}}
  {2}\Big)^{p}\big(u^*(O^{-t}O^t\nabla_{Ox} u)\big)^{1-p}e^{-O u }dx\\
  &=2p \int_{\mathbb R^n}\log\frac{\|TO^t\nabla u\|}{\|\nabla u\|}\Big(\frac{\|\nabla u\|^{2}}
  {2}\Big)^{p}\big(u^*(\nabla u)\big)^{1-p}e^{-u}dx\\
  \nonumber&=2p\int_{\mathbb R^n}\log\frac{\|TO^tx\|}{\|x\|}\Big(\frac{\|x\|^{2}}
  {2}\Big)^{p}d\mu_p(f,x).
\end{align}
Then we have
\begin{align}\label{p-geq-0}
 \nonumber\int_{\mathbb R^n}\log&\Big(\frac{\|\phi_1^{-t}x\|^2}{\|x\|^2}\Big)^p\Big(\frac{\|x\|^{2}}{2}\Big)^pd\mu_p(f,x)\\
 \nonumber&=2p\int_{\mathbb R^n}\log\frac{\|{TO^t}x\|}{\|x\|}\Big(\frac{\|x\|^2}{2}\Big)^pd\mu_p(f,x)\\
 \nonumber&=2p\int_{\mathbb R^n}\log\Big(\frac{\|T x \|}{\|x\|}\Big)\Big(\frac{\|x\|^{2}}
  {2}\Big)^{p}d\mu_p(O f,x)\\
&=p\int_{\mathbb R^n}\log\Big(\sum_{i=1}^n{x}_i^2\lambda_i^{2}\Big)\Big(\frac{\|x\|^{2}}{2}\Big)^{p}d\mu_p( O f,x)\\
 \nonumber&\geq p\int_{\mathbb R^n}\Big[\sum_{i=1}^nx_i^2\log\lambda_i^{2}\Big]\Big(\frac{\|x\|^{2}}
  {2}\Big)^{p}d\mu_p(O f,x)\\
   \nonumber&=2p\delta J_p(Of,\gamma)\sum_{i=1}^n\log\lambda_i=0.
\end{align}
Where $x_i=\langle\frac{x}{\|x\|}\cdot e_i\rangle$. Then we have
\begin{align}\label{p-equ-1}
  \Big[\frac{1}{p \delta J_p(f,\gamma)}&\int_{\mathbb R^n}h^p_{\gamma_{\phi_1}}(x)d\mu_p(f,x)\Big]^{\frac{1}{p}}\\
   \nonumber&\geq exp\Big[\frac{2}{\delta J_p(f,\gamma)}\int_{\mathbb R^n}\log\frac{\|{\phi_1^{-t}}x\| }{\|x\|}\Big(\frac{\|x\|^{2}}{2}\Big)^pd\mu_p(f,x)\Big]^{\frac{1}{p}}\\
   \nonumber&\geq 1.
\end{align}
The first inequality in (\ref{p-equ-1}) is a consequence of the Jensen's inequality, with equality holds if and only if there exist a constant $c>0$ such that $\frac{\|\phi_1^{-t}x\|}{\|x\|}=c$ for all $x\in\mathbb R^n$.

Moreover, note that from the strict concave of the log-concave function that equality in (\ref{p-geq-0}) is possible only if $x_1\cdots x_n\neq 0$ which implies $\lambda_1=\cdots=\lambda_n$, for $x\in \mathbb R^n$.  Thus $\|Tx\|=\lambda_i$ when $x_i\neq 0$. Now  equality in (\ref{p-equ-1}) would force $\frac{\|\phi_1^{-t}x\|}{\|x\|}=c$ or equivalently $\|T^{-t}x\|=c$ for $x\in\mathbb R^n$, so that $\lambda_i=c$ for all $i$. This together with the fact $\lambda_1\cdots\lambda_n=1$ shows that equality in (\ref{p-equ-1}) would imply $T=I$ and hence $\phi_1=I$.
\end{proof}

\begin{thm}\label{suf-sol-sp}
  Let $f\in\mathcal A_0$. Then problem $S_p$ has a unique solution. Moreover a Gaussian function $\gamma_{\phi}$ solves $S_p$  if and only if it satisfies
  \begin{align}\label{equi-solut-oversp-1}
\Big(\frac{c_n}{J(\gamma_\phi)}\Big)^{\frac{2}{n}}\delta J_p(f,\gamma_{\phi})h_{\gamma}(\sqrt \alpha\phi^{-t}x)(y)=\frac{n}{4p}\int_{\mathbb R^n}|\langle x,y\rangle|^2h_{\gamma}(\sqrt \alpha\phi^{-t}y)^{p-1}d\mu_p(f,x).
\end{align}
for any $y\in \mathbb R^n$.
\begin{proof}
By computation shows $\alpha^p\cdot_p e^{-v}=e^{-\alpha v(\frac{x}{\alpha})}$, then
  $\alpha^p\cdot_p\gamma_\phi=\gamma_{\frac{\phi}{\sqrt\alpha}}.$
So we obtain
  $h_{\big(\alpha^p\cdot_p\gamma_\phi\big)^o}=h_{\gamma}(\sqrt \alpha\phi^{-t}x)$.

Together with Lemma \ref{ove-S to S} and Theorem \ref{suf-sol-ove-sp} we have
  \begin{align*}
\Big(\frac{c_n}{J(\gamma_\phi)}\Big)^{\frac{2}{n}}\delta J_p(f,\gamma_{\phi})h_{\gamma}(\sqrt \alpha\phi^{-t}y)=\frac{n}{4p}\int_{\mathbb R^n}|\langle x,y\rangle|^2h_{\gamma}(\sqrt \alpha\phi^{-t}y)^{p-1}d\mu_p(f,x).
\end{align*}
\end{proof}
\end{thm}

Now we define $L_p$ John ellipsoid for log-concave functions.
\begin{defn}\label{lp john-ell-fun}
  Let $f\in \mathcal A_0$ be log-concave function, the unique Gaussian function that solves the constrained maximization problem
  \begin{align}
    \max J(\phi\gamma)\,\,\,\,\,\,\,\,\,\,\mbox{subject to}\,\,\,\,\,\,\,\,\,\, \overline\delta J_p(f,\phi\gamma)\leq 1,
  \end{align}
  is called the $L_p$ John ellipsoid of $f$ and denoted by $E_pf$. The unique Gaussian function that solves the constrained minimization problem
  \begin{align}
    \min\overline\delta J_p(f,\phi\gamma) \,\,\,\,\,\,\,\,\mbox{subject to}\,\,\,\,\,\,\,\,\,\, \frac{J(\phi\gamma)}{c_n}= 1,
  \end{align}
  is called the normalized $L_p$ John ellipsoid of $f$ and denoted by $\overline E_pf$.
\end{defn}

Specially, if we take $f=e^{-\frac{\|x\|^2_K}{2}}$ for $K\in \mathcal K^n_o$, since the Gaussian function $\gamma_\phi$ can be viewed as $\gamma_\phi=e^{-\frac{\|x\|_E^2}{2}}$, where $E$ is the origin-centered ellipsoid. Then Definition \ref{lp john-ell-fun} deduce the definition of $L_p$ John ellipsoid defined in \cite{lut-yan-zha-lp2005}.

Since $\overline\delta J_p(\phi f, g)=\overline\delta J_p(f,\phi^{-1}g)$ for $\phi\in GL(n)$, then we have the following result.
\begin{prop}\label{gln-inv}
  Let $f\in\mathcal A_0$. Then
  \begin{align}
    E_p(T f)=T(E_pf).
  \end{align}
\end{prop}
\begin{proof}
 By the definition of problem $S_p$, set $E_p f=\gamma_\phi$, since $\overline \delta J_p(f,\gamma_\phi)=1$.
  By the Lemma \ref{inv-det-j}, we have
  \begin{align}
\overline \delta J_p(Tf,T\gamma_\phi)=\overline \delta J_p(f,T^{-1}T\gamma_\phi)=\overline \delta J_p(f,\gamma_\phi)=1.
  \end{align}
  By the uniqueness of the problem $S_p$, we have $E_p(T f)=T (E_pf)$. So we complete the proof.
\end{proof}
If we take $f=\gamma$, then we have $E_p\gamma=\gamma$. Moreover by Proposition \ref{gln-inv}, we have the following.
\begin{cor}
Let  $\gamma_\phi=e^{-\frac{\|\phi x\|^2}{2}}$, for $\phi\in GL(n)$, then
$$E_p(\gamma_\phi)=\gamma_\phi.$$
\end{cor}

\section{Continuity}
In this section, we will show that the family of $L_p$ John ellipsoids for log-concave functions is continuous in $p\in [1,\infty).$ First let $f=e^{-u}\in \mathcal A_0$, note that if  $\overline E_pf$ is a solution of problem $\overline S_p$, then there exist a $\overline\phi\in SL(n)$ such that $\overline E_pf=\gamma_{\overline\phi}=e^{-\frac{\|\overline\phi x\|^2}{2}}$.

%
%
%

\begin{thm}\label{cont-fun-john}
  Let $f$ and $\{f_i\}$ in $\mathcal A_0$, such that $\lim_{i\rightarrow \infty}f_i=f$ on $\mathbb R^n$. Then
  \begin{align}
    \lim_{i\rightarrow \infty}\overline E_pf_i=\overline E_pf.
  \end{align}
  \end{thm}

To prove Theorem \ref{cont-fun-john}, we need the following Theorem.

First we prove that $d\mu_p(f_i,x)\rightarrow d\mu_p(f,x)$
\begin{lem}\label{mea-conv}
  Let $f=e^{-u},\,f_i=e^{-u_i}\in \mathcal A_0$, if $f_i\rightarrow f$, then $d\mu_p(f_i,x)\rightarrow d\mu_p(f,x)$.
\end{lem}
\begin{proof}
We only need to prove that, for any function $g\in L^1(\mu_p(f,\cdot))$,
\begin{align}
\lim_{i\rightarrow \infty}\int_{\mathbb R^n}g(x)d\mu_p(f_i,x) = \int_{\mathbb R^n}g(x)d\mu_p(f,x).
\end{align}
Set $a=\max\{|g(x)|:x\in\mathbb R^n\}$, $b=\max\{|f(x)|:x\in\mathbb R^n\}$ $b_i=\max\{|f_i(x)|:x\in\mathbb R^n\}$, $c_i=\max\{|h_{f_i}(x)|:x\in\mathbb R^n\}$, $c=\max\{|h_f(x)|: x\in \mathbb R^n\}$. Note that $f$ and $f_i\in\mathcal A_0$ are integrable, and  $h_{f_i}\rightarrow h_{f}$ whenever $f_i\rightarrow f$, then there exist an $N_1\in\mathbb N$ such that
$|f_i-f|<\frac{\epsilon}{2ac_i^p}$ for $i\geq N_1$, and $N_2\in\mathbb N$ such that $|h_{f_i}-h_f|<\frac{\epsilon}{2ab\sum_{j=0}^{p-1}c^{j}c_i^{p-1-j}}$ for $i\geq N_2$.  Since $d\mu_p(f,x)=h^{1-p}_ffdx$, so we can choose $i\geq \max\{N_1,\,N_2\}$, then
\begin{align*}
\Big|\int_{\mathbb R^n}gd\mu_p(f_i,x)&-\int_{\mathbb R^n}gd\mu_p(f,x)\Big|\leq\int_{\mathbb R^n} \big|gh^p_{f_i}f_i-gh^p_{f}f\big|dx,\\
&\leq \int_{\mathbb R^n} |gh^p_{f_i}||f_i-f|dx+\int_{\mathbb R^n}|gf||h^p_{f_i}-h^p_f|dx.\\
&\leq \epsilon.
\end{align*}
So we complete the proof.
\end{proof}

\begin{thm}\label{cont-delt}
  Suppose $f=e^{-u},\,f_i=e^{-u_i}, \,g=r^{-v},\,\,g_j=e^{-v_j}\in \mathcal A_0$, where $i,\,j\in \mathbb N$. If $f_i\rightarrow f$, $g_i\rightarrow g$, then
\begin{align}
  \lim_{i,j\rightarrow\infty}\delta J_p(f_i,g_j)=\delta J_p(f,g).
\end{align}
 \end{thm}

\begin{proof}
 Since $f_i\rightarrow f$, $g_i\rightarrow g$, by the Definition of $\delta J_p(f,g)$, we have
\begin{align}\label{fra-1}
  &\Big|\int_{\mathbb R^n}h^p_{g_i}d\mu_p(f_i,x)-\int_{\mathbb R^n}h^p_gd\mu_p(f,x) \Big|\\
\nonumber &\leq  \Big|\int_{\mathbb R^n}h^p_{g_i}d\mu_p(f_i,x)-\int_{\mathbb R^n}h^p_{g_i}d\mu_p(f,x) \Big|
+ \int_{\mathbb R^n}|h^p_{g_i}-h^p_g|d\mu_p(f,x).
\end{align}
Note that $h_g$ and $h_{g_i}$ are bounded, set $c_i=\max\{|g_i(x)|:x\in\mathbb R^n\}$ and $c_0=\max\{|g(x)|:x\in\mathbb R^n\}$, then $c_i,\,\,c_0$  are bounded. Let $c=\sum_{j=1}^{p-1}c_i^jc_0^{p-1-j}$ and $|h_{g_i}-h_g|\leq \frac{\epsilon}{2c}$, then
\begin{align}\label{fra-2}
\int_{\mathbb R^n}\big|h^p_{g_i}-h^p_g|d\mu_p(f,x)\leq \frac{\epsilon}{2}.
\end{align}
By Lemma \ref{mea-conv}, we can show that
\begin{align}\label{fra-3}
  \Big|\int_{\mathbb R^n}h^p_{g_i}d\mu_p(f_i,x)-\int_{\mathbb R^n}h^p_{g_i}d\mu_p(f,x) \Big|\leq \frac{\epsilon}{2}.
\end{align}
So together with formulas (\ref{fra-1}), (\ref{fra-2}) and (\ref{fra-3}) we can choose unified $N$ such that
\begin{align*}
  &\Big|\int_{\mathbb R^n}h^p_{g_i}d\mu_p(f_i,x)-\int_{\mathbb R^n}h^p_gd\mu_p(f,x) \Big|\leq\epsilon.
\end{align*}
We complete the proof.
\end{proof}

Now we give a proof of Theorem \ref{cont-fun-john}
\begin{proof}
  [Proof of Theorem \ref{cont-fun-john}] By Theorem \ref{cont-delt}, and $J(f^\diamond_i)\rightarrow J(f^\diamond)$ when $i\rightarrow \infty$, then we have
\begin{align*}
  \lim_{i\rightarrow \infty}\overline\delta J_p(f_i, \overline E_pf_i)&=\lim_{i\rightarrow \infty}\min_{J(\gamma_\phi)=c_n}\overline\delta J_p(f_i, \gamma_\phi)\\
&=\min_{J(\gamma_\phi)=c_n}\lim_{i\rightarrow \infty}\overline\delta J_p(f_i, \gamma_\phi)\\
&=\min_{J(\gamma_\phi)=c_n}\overline\delta J_p(f, \gamma_\phi)\\
&=\overline\delta J_p(f, \overline E_pf).
\end{align*}
We complete the proof.
\end{proof}

Note that $\overline E_pf$ is bounded for $p\in [0,\infty]$. Thus in order to establish the continuity of $\overline E_pf$ in $p\in [0,\infty]$, (\ref{lip-ineq}) shows that the $\overline \delta J_p(K,\cdot)$ is continuity of $p\in [1,\infty]$.
\begin{lem}
  If $p_0\in [1,\infty]$, then
\begin{align}
  \lim_{p\rightarrow p_0}\overline\delta J_p(f,\gamma_{\overline \phi})=\overline\delta J_{p_0}(f,\gamma_{\overline \phi}).
\end{align}
for some $\overline \phi \in SL(n)$.
\end{lem}

\begin{thm}\label{p-q-des}
Let $f=e^{-u }\in\mathcal A_0$, and $1\leq p\leq q< \infty$, $E_pf$  be the solution of  constrained maximization problem, then
\begin{align}
 J(E_\infty f)\leq J(E_qf)\leq J(E_pf)\leq J(E_1f).
\end{align}
\end{thm}
\begin{proof}
 The definition of $\overline \delta J_p(f, g)$, together with Jensenn's inequality, for $1\leq p\leq q< \infty$, we have
  \begin{align*}
 \overline \delta J_p(f, \gamma)&=\Big[\frac{1}{J(f^\diamond)}\int_{\mathbb R^n}\Big(\frac{h_{\gamma}}{h_f}\Big)^p
  h_fd\mu(f,x)\Big]^{\frac{1}{p}}\\
  &\leq\Big[\frac{1}{ J(f^\diamond)}\int_{\mathbb R^n}\Big(\frac{h_{\gamma}}{h_f}\Big)^q
  h_fd\mu(f,x)\Big]^{\frac{1}{q}}\\
  &= \overline \delta J_q(f, \gamma).
  \end{align*}
By Definition \ref{lp john-ell-fun}, we have
\begin{align*}
  E_qf=\max\big\{E'f:\overline \delta J_q(f, Ef)\leq 1\big\}\leq \max\big\{E'f:\overline \delta J_q(f, Ef)\leq 1\big\}=E_pf.
\end{align*}
This implies $J(E_qf)\leq J(E_pf)$. For $p\rightarrow\infty$, by the definition of (\ref{inf-delta-p}), and the continuity of $p\in[1,\infty]$, we have $E_\infty f =\lim_{p\rightarrow \infty}E_pf,$
we complete the proof.
\end{proof}
\begin{thm}\label{joh-leq-ine}
Let $f=e^{-u }\in\mathcal A_0$, such that $J(f)>0$. Let  $1\leq p< \infty$,  $E_pf$  be the solution of constrained maximization problem, then
\begin{align}
  J(E_pf)\leq J(f).
\end{align}
\end{thm}
\begin{proof}
By the definition \ref{lp john-ell-fun}, we have
\begin{align*}
  1=\overline \delta J_p(f,E_pf)^p&=\frac{p\cdot\delta J_p(f,E_pf)}{J(f^\diamond)}.
  \end{align*}
 The $L_p$ Minkowski inequality for log-concave functions (see \cite{fan-xin-ye-geo2020}) says that
 \begin{align*}
  \frac{J(f^\diamond)}{p}&\geq\delta J_p(f,f)+J(f)\log  \frac{J(E_pf)}{J(f)}.
\end{align*}
This means that
\begin{align*}
J(f)\log\frac{J(E_pf)}{J(f)}\leq 0.
\end{align*}
Sine $J(f)>0$, then $\log\frac{J(E_pf)}{J(f)}\leq 0.$
That means
$$J(E_pf)\leq J(f).$$
\end{proof}
 The functional Blaschke-Santal\'o inequality, proved for even functions in \cite{bal-iso1986}, and given in full generality in  \cite{avi-kla-mil-the2004} says that for log-concave function $f\in\mathcal A_0$, the following inequality holds
  $$P(f)\leq P(\gamma),$$
  that is
  $$J(f)J(f^o)\leq J(\gamma)^2=(2\pi)^n.$$
  So combine with the Theorem \ref{joh-leq-ine}, we have the following Theorem.
\begin{thm}
Let $f=e^{-u }\in\mathcal A_0$, such that $J(f)>0$. Let  $1\leq p< \infty$,  $E_pf$  be the solution of constrained maximization problem, then
\begin{align}
  J(E_pf)J(E_pf^o)\leq c_n^2.
\end{align}
where $c_n=(2\pi)^\frac{n}{2}$.
\end{thm}

In the following, denoting by $\triangle_n$ and $B^n_\infty$  the regular simplex centered at the origin and the unit cube in $\mathbb R^n$. To establish the $L_p$ Ball's ratio inequality for log-concave function, we need the following results, more details see \cite{gut-mer-jim-vil-joh2018}.
\begin{thm}[\cite{gut-mer-jim-vil-joh2018}]\label{bal-rat-ine-log}
  Let $f\in \mathcal A_0$, $Ef$ be the functional John ellipsoid of $f$, then
\begin{align}
\frac{J(f)}{J(Ef)}\leq \frac{J(g_c)}{J(E g_c )},
\end{align}
where $g_c(x)=e^{-\|x\|_{\triangle_n-c}}$ for any $c\in \triangle_n$. Furthermore , there is equality if and only if $\frac{f}{\|f\|_\infty}=T g_c$ for some affine map $T$ and some $c\in \triangle_n$. If we assume $f$ to be even, then
\begin{align}
\frac{J(f)}{J(Ef)}\leq \frac{J(g)}{J(E g )},
\end{align}
where $g(x)=e^{-\|x\|_{B_\infty^n}}$, with equality if and only if $\frac{f}{\|f\|_\infty}=T g$ for some linear map $T\in GL(n)$.
\end{thm}
Moreover, by compute the right hand of the above formulas, it gives
\begin{lem}[\cite{gut-mer-jim-vil-joh2018}]
  Let $f\in\mathcal A_0$, $Ef$ be the functional John ellipsoid of $f$, then
\begin{align}
 \frac{J(f)}{J(Ef)}\leq \frac{e}{n}(n!)^{\frac{1}{n}}\frac{|\triangle_n|}{|E\triangle_n|},
\end{align}
If we assume $f$ to be even, then
\begin{align}
 I. rat(f)\leq \frac{e}{n}(n!)^{\frac{1}{n}}\frac{|B^n_\infty|}{|E_{B^n_\infty}|}.
\end{align}
\end{lem}

Now we give the Ball's volume ration inequality for log-concave function.
\begin{thm}
 Let $f=e^{-u }\in\mathcal A_0$, such that $J(f)>0$. Let  $1\leq p< \infty$,  $E_pf$  be the solution of constrained maximization problem, then
 \begin{align}
  \frac{J(f)}{J(E_pf)}\leq\frac{n^{\frac{n-2}{n}}(n+1)^{\frac{n+1}{2}}e }{(n!)^{\frac{n-1}{n}}\omega_n}.
\end{align}
 If $f$ is even then
\begin{align}
  \frac{J(f)}{J(E_pf)}\leq \frac{e}{n}(n!)^{\frac{1}{n}}\frac{2^n}{\omega_n}.
\end{align}
where $\omega_n$ is the volume of ball $B^n$.
\end{thm}
\begin{proof}
 By Theorem \ref{p-q-des}, and the fact $E_\infty f=Ef$, then we have
\begin{align*}
  \frac{J(f)}{J(E_pf)}\leq \frac{J(f)}{J(Ef)}\leq \frac{e}{n}(n!)^{\frac{1}{n}}\frac{|\triangle_n|}{|J\triangle_n|}.
\end{align*}
On the other hand, note that the volume of $\triangle_n$ is given by
$|\triangle_n|=\frac{\sqrt{n+1}}{2^{\frac{n}{2}}n!}$, and the inradius of $\triangle_n$ is given by
$r_{\triangle_n}=\frac{1}{\sqrt{2n(n+1)}}$. Then by a simple computation gives
\begin{align*}
  \frac{J(f)}{J(E_pf)}\leq \frac{e}{n}(n!)^{\frac{1}{n}}\frac{|\triangle_n|}{|J\triangle_n|}=\frac{ n^{\frac{n-2}{n}}(n+1)^{\frac{n+1}{2}}e}{(n!)^{\frac{n-1}{n}}\omega_n}.
\end{align*}
If $f$ is even, then
\begin{align*}
  \frac{J(f)}{J(E_pf)}\leq \frac{e}{n}(n!)^{\frac{1}{n}}\frac{|B^n_\infty|}{|JB^n_\infty|} = \frac{e}{n}(n!)^{\frac{1}{n}}\frac{2^n}{\omega_n}.
\end{align*}
So we complete the proof.
\end{proof}

{\bf Acknowledgments}\\

The authors would like to strongly thank the anonymous referee for the very valuable comments and helpful suggestions that directly lead to improve the original manuscript.

\bibliographystyle{amsplain}

\begin{thebibliography}{10}

\bibitem{alo-a2019}
D.~Alonso-Guti\'errez, \emph{{A reverse Rogers-Shephard inequality for
  log-concave functions}}, J. Geom. Anal. \textbf{29} (2019), 299--315.

\bibitem{alo-ber-mer-an2020}
D.~Alonso-Guti\'errez, J.~Bernu\'es, and B.~Merino, \emph{{An extension of
  Berwald's inequality and its relation to Zhang's inequality}}, J. Math. Anal.
  Appl. \textbf{486} (2020), 123875.

\bibitem{alo-mer-jim-vil-rog2016}
D.~Alonso-Guti\'errez, B.~Merino, C.~Jim\'enez, and R.~Villa,
  \emph{{Roger-Shephard inequality for log-concave functions}}, J. Funct. Anal.
  \textbf{271} (2016), 3269--3299.

\bibitem{alo-mer-jim-vil-joh2018}
\bysame, \emph{{John's ellipsoid and the integral ratio of a log-concave
  function}}, J. Geom. Anal. \textbf{28} (2018), 1182--1201.

\bibitem{gut-mer-jim-vil-joh2018}
D.~Alonso-Guti\'{e}rrez, B.~Merino, C.~Jim\'{e}nez, and R.~Villa, \emph{{John's
  Ellipsoid and the Integral Ratio of a Log-Concave Function}}, J. Geom. Anal.
  \textbf{28} (2018), 1182--1201.

\bibitem{avi-kla-mil-the2004}
A.~Artstein-Avidan, B.~Klartag, and V.~Milman, \emph{{The santal\'{o} point of
  a function, and a function form of the Santal\'o inequality}}, Mthematika
  \textbf{51} (2004), 33--48.

\bibitem{art-flo-seg-fun2020}
S.~Artstein-Avidan, D.I. Florentin, and A.~Segal, \emph{{Functional
  Brunn-Minkowski inequalities induced by polarity}}, Adv. Math. \textbf{364}
  (2020), 107006.

\bibitem{avi-mil-the2009}
S.~Artstein-Avidan and V.~Milman, \emph{{The concept of duality in convex
  analysis, and the characterization of the Legendre transform}}, Ann. Math.
  \textbf{169} (2009), no.~2, 661--674.

\bibitem{avi-mil-a2010}
\bysame, \emph{{A characterization of the support map}}, Adv. Math.
  \textbf{223(1)} (2010), 379--391.

\bibitem{avi-kla-sch-wer-fun2012}
S.~Avidan, B.~Klartag, C.~Sch\"{u}tt, and E.~Werner, \emph{{Functional
  affine-isoperimetry and an inverse logarithmic Sobolev inequality}}, J.
  Funct. Anal. \textbf{262} (2012), 4181--4204.

\bibitem{bal-iso1986}
K.~Ball, \emph{{Isometric Problems in $l_p$ and Section of Convex Sets}}, Ph.D
  dissertation, Cambridge, 1986.

\bibitem{bal-vol1989}
\bysame, \emph{{Volumes of sections of cubes and related problems. In:
  Lindenstrauss J., Milman V.D. (eds) Geometric Aspects of Functional
  Analysis}}, Lecture Notes in Mathematics, vol. 1367, Springer, Berlin,
  Heidelberg, 1989.

\bibitem{bal-vol1991}
\bysame, \emph{{Volume ratios and a reverse isoperimetric inequality}}, J.
  London Math. Soc. \textbf{(2)44} (1991), 351--359.

\bibitem{bal-ell1992}
\bysame, \emph{{Ellipsoids of maximal volume in convex bodies}}, Geom. Dedicata
  \textbf{(2)41} (1992), 241--250.

\bibitem{bra-lie-on1976}
H.~Brascamp and E.~Lieb, \emph{{On extensions of the Brunn-Minkowski and
  Pr\'{e}kopa-Leindler theorems,including inequalities for log concave
  functions, and with an application to diffusion equation}}, J. Funct. Anal.
  \textbf{22} (1976), 366--389.

\bibitem{cag-fra-gue-leh-sch-wer-fun2016}
U.~Caglar, M.~Fradelizi, O.~Gu\'edon, J.~Lehec, C.~Sch\"uett, and E.~Werner,
  \emph{{Functional versions of $L_p$-affine surface area and entropy
  inequalities}}, Int. Math. Res. Not. \textbf{4} (2016), 1223--1250.

\bibitem{cag-wer-div2014}
U.~Caglar and E.~Werner, \emph{{Divergence for s-concave and log concave
  functions}}, Adv. Math. \textbf{257} (2014), 219--247.

\bibitem{cag-wer-mix2015}
\bysame, \emph{{Mixed f-divergence and inequalities for log concave
  functions}}, Proc. Lond. Math. Soc. \textbf{110(2)} (2015), 271--290.

\bibitem{cal-ye-aff2016}
U.~Caglar and D.~Ye, \emph{{Affine isoperimetric inequaities in the functional
  Orlicz-Brunn-Minkowski theory}}, Adv. Appl. Math. \textbf{81} (2016),
  78--114.

\bibitem{cag-ye-aff2016}
\bysame, \emph{{Affine isoperimetric inequalities in the functional
  Orlicz-Brunn-Minkowski theory}}, Adv. Appl. Math. \textbf{81} (2016),
  78--114.

\bibitem{cho-wan-the2006}
K.~Chou and X.~Wang, \emph{{The $L_p$-Minkowski problem and the Minkowski
  problem in centroaffine geometry}}, Adv. Math. \textbf{205} (2006), 33--83.

\bibitem{col-bru2005}
A.~Colesanti, \emph{{Brunn-Minkowski inequalities for variational functionals
  and related problems}}, Adv. Math. \textbf{257} (2005), 219--247.

\bibitem{col-fun2006}
\bysame, \emph{{Functional inequality related to the Rogers-Shephard
  inequality}}, Mathematika \textbf{53} (2006), 81--101.

\bibitem{col-fra-the2013}
A.~Colesanti and I.~Fragal\`{a}, \emph{{The first variation of the total mass
  of log-concave functions and related inequalities}}, Adv. Math. \textbf{244}
  (2013), 708--749.

\bibitem{era-kla-mom2015}
D.~Erasusquin and B.~Klartag, \emph{{Moment measure}}, J. Funct. Anal.
  \textbf{268} (2015), 3834--3866.

\bibitem{fan-xin-ye-geo2020}
N.~Fang, S.~Xing, and D.~Ye, \emph{{Geometry of log-concave functions: the
  $L_p$ Asplund sum and the $L_p$ Minkowski problem}}, arXiv preprint arXiv:
  2006. 16959 (2020).

\bibitem{fan-zho-lyz2018}
N.~Fang and J.~Zhou, \emph{{LYZ ellipsoid and Petty projection body for
  log-concave function}}, Adv. Math. \textbf{340} (2018), 914--959.

\bibitem{fle-gue-pao-a2007}
B.~Fleury, O.~Gu\'{e}don, and G.~Paouris, \emph{{A stability result for mean
  width of $L_p$-centrodi bodies}}, Adv. Math. \textbf{214} (2007), 865--877.

\bibitem{fra-mey-som2007}
M.~Fradelizi and M.~Meyer, \emph{{Some functional forms of Blaschke-Santal\'o
  inequality}}, Math. Z. \textbf{256} (2007), 379--395.

\bibitem{gar-the2002}
R.~Gardner, \emph{The {Brunn-Minkowski inequality}}, Bull. Amer. Math. Soc.
  \textbf{39} (2002), 355--405.

\bibitem{gar-geo2006}
\bysame, \emph{{Geometric Tomography}}, 2nd edition, {Encyclopedia Math.
  Appl.}, vol.~58, Cambridge University Press, Cambridge, 2006.

\bibitem{gia-mil-ext2000}
A.~Giannopoulos and V.~Milman, \emph{Extremal problem and isotropic position of
  convex bodies}, Israel. J. Math. \textbf{117} (2000), 29--60.

\bibitem{gru-con2007}
P.~Gruber, \emph{{Convex and Discrete Geometry}}, Springer, Berlin Heidelberg,
  2007.

\bibitem{gru-joh2011}
\bysame, \emph{{John and Loewner ellipsoids}}, Discrete Comput. Geom.
  \textbf{46} (2011), 776--788.

\bibitem{gut-a2019}
D.~Guti\'errez, \emph{{A reverse Rogers Shephard inequality for log-concave
  functions}}, J. Geom. Anal. \textbf{29} (2019), 299--315.

\bibitem{hab-lut-yan-zha-the2010}
C.~Haberl, E.~Lutwak, D.~Yang, and G.~Zhang, \emph{{The even Orlicz Minkowski
  problem}}, Adv. Math. \textbf{224} (2010), 2485--2510.

\bibitem{hab-sch-gen2009}
C.~Haberl and F.~Schuster, \emph{{General $L_p$ affine isoperimetric
  inequalities}}, J. Differential Geom. \textbf{83} (2009), 1--26.

\bibitem{kla-on2004}
B.~Klartag, \emph{{On John-type ellipsoids}}, Geometric Aspects of Fnctional
  Analysis. Lecture Notes in Mathematics \textbf{1850} (2004), 149--158.

\bibitem{lew-ell1979}
D.~Lewis, \emph{{Ellipsoid defined by Banach ideal norms}}, Mathematika
  \textbf{26} (1979), 18--29.

\bibitem{li-sch-wer-the2019}
B.~Li, C.~Sch\"utt, and E.~Werner, \emph{{The L\"owner Function of a
  Log-Concave Function}}, J. Geom. Anal. (2019), 1--34.

\bibitem{lin-aff2017}
Y.~Lin, \emph{{Affine Orlicz P\'{o}lya-Szeg\"{o} for log-concave functions}},
  J. Funct. Anal. \textbf{273} (2017), 3295--3326.

\bibitem{lud-gen2010}
M.~Ludwig, \emph{{General affine surface areas}}, Adv. Math. \textbf{224}
  (2010), 2346--2360.

\bibitem{lut-the1993}
E.~Lutwak, \emph{The {Brunn-Minkowski-Firey Theory I: Mixed volumes and the
  Minkowski} problem}, J. Differential Geom. \textbf{38} (1993), 131--150.

\bibitem{lut-the1996}
\bysame, \emph{The {Brunn-Minkowski-Firey Theory II: Affine and geominimal
  surface area}}, Adv. Math. \textbf{118} (1996), 224--194.

\bibitem{lut-yan-zha-lp2000}
E.~Lutwak, D.~Yang, and G.~Zhang, \emph{{$L_p$} affine isoperimetric
  inequalities}, J. Differential Geom. \textbf{56} (2000), 111--132.

\bibitem{lut-yan-zha-lp2005}
\bysame, \emph{{$L_p$ John ellipsoids}}, Proc. Lond. Math. Soc. \textbf{90}
  (2005), 497--520.

\bibitem{pao-wer-rel2012}
G.~Paouris and E.~Werner, \emph{{Relative entropy of cone measures and $L_p$
  centroid bodies}}, Proc. Landon Math. Soc. \textbf{104} (2012), 253--286.

\bibitem{pis-the1989}
G.~Pisier, \emph{The colume of convex bodies and banach space geometry},
  Cambridge University Press, 1989.

\bibitem{pre-new1975}
A.~Pr\'{e}kopa, \emph{{New proof for the basic theorem of logconcave
  measures}}, Alkalmaz. Mat. Lapok \textbf{1} (1975), 385--389.

\bibitem{roc-con1970}
T.~Rockafellar, \emph{{Convex analysis}}, Princeton Press, Princeton, 1970.

\bibitem{rot-on2012}
L.~Rotem, \emph{{On the mean width of log-concave functions}}, in: Geometric
  Aspects of Functional Analysis, in: Geometric Aspects of Functional Analysis,
  in: Lecture Notes in Math., vol. 2050, Springer, Berlin, 2012.

\bibitem{rot-sup2013}
\bysame, \emph{{Support functions and mean width for $\alpha$-concave
  functions}}, Adv. Math. \textbf{243} (2013), 168--186.

\bibitem{sch-con1993}
R.~Schneider, \emph{{ Convex Bodies: The Brunn-Minkowski Theory}}, Encyclopedia
  of Mathematics and Its Applications, vol.~44, Cambridge University Press,
  Cambridge, 1993.

\bibitem{sch-wer-sur2004}
C.~Sch\"{u}tt and E.~Werner, \emph{{Surface bodies and p-affine surface area}},
  Adv. Math. \textbf{187} (2004), 98--145.

\bibitem{sta-the2002}
A.~Stancu, \emph{{The discrete planar $L_0$-Minkwoski problem}}, Adv. Math.
  \textbf{167} (2002), 160--174.

\bibitem{sta-cen2012}
\bysame, \emph{{Centro-affine invariants fot smooth conve bodies}}, Int. Math.
  Res. Not. IMRN \textbf{10} (2012), 2289--2320.

\bibitem{wer-ren2012}
E.~Werner, \emph{{Renyi Divergence and $L_p$ affine surface area for convex
  bodies}}, Adv. Math. \textbf{230} (2012), 1040--1059.

\bibitem{wer-ye-new2008}
E.~Werner and D.~Ye, \emph{{New $L_p$ affine isoperimetric inequalities}}, Adv.
  Math. \textbf{218} (2008), 762--780.

\bibitem{zou-xio-orl2014}
D.~Zou and G.~Xiong, \emph{{Orlicz-John ellipsoids}}, Adv. Math. \textbf{265}
  (2014), 132--168.

\end{thebibliography}

\end{document}